\documentclass{rmmcart}

\usepackage{latexsym, amssymb, amsthm,amsmath}

\theoremstyle{theorem}
\newtheorem{theorem}{Theorem}
\newtheorem{lemma}[theorem]{Lemma}
\newtheorem{corollary}[theorem]{Corollary}

\theoremstyle{definition}
\newtheorem*{definition}{Definition}
\newtheorem*{definitions}{Definitions}
\newtheorem{remark}[theorem]{Remark}
\newtheorem*{conj}{Conjecture}
\newtheorem*{notation}{Notation}
\newtheorem*{example}{Examples}
\newtheorem*{remark*}{Remark}

\title{Down the Large Rabbit Hole}

\author{Aaron Robertson}
\address{Department of Mathematics, Colgate University, Hamilton, New York}
\email{arobertson@colgate.edu}
\date{August 2017}



\begin{document}

\begin{abstract}
This article documents my journey down the rabbit hole,
chasing what I have come to know as a particularly unyielding
problem in Ramsey theory on the integers: the $2$-Large Conjecture.
This conjecture states that if $D \subseteq \mathbb{Z}^+$ has the
property that every $2$-coloring of $\mathbb{Z}^+$ admits arbitrarily
long monochromatic arithmetic progressions with common difference
from $D$ then the same property holds for any finite number of
colors.  We hope to provide a roadmap for future researchers and
also provide some new results related to the $2$-Large Conjecture.
\end{abstract}

\maketitle

\section{Prologue}

Mathematicians tend not to write of their failures.  This is rather 
unfortunate as there are surely countless creative ideas that have never
seen the light of day; I have long believed that a \textsc{Journal of
Failed Attempts} should exist.  My goal with this article is 3-fold:
(1) a chronicle of my battle with what I consider a particularly difficult
conjecture; (2) to present my progress on this conjecture; and (3) to
provide a roadmap to those who want to take on this challenging conjecture.

The majority of this work took place over the course of a year, circa 
2010\footnote{Supported in part by the National Security Agency
[grant number H98230-10-1-0204].}. %, Grant \#H98230-10-1-0204.}  
Since that time I have frequently revisited this intriguing problem, even though
that year was mostly an exercise in banging my head against various brick walls.  I wish I knew how to quit it.  I love this
conjecture, so much so that I've followed it down the rabbit hole.
However, if we are to take away one message from Steinbeck's
{\it Of Mice and Men}, it's that sometimes the rabbit doesn't love you back.

\section{What's Up, Doc?}

Ramsey theory may best be summed up as ``the study of the preservation of structures under set partitions" \cite{LR}.
For this article, we will restrict our attention to the positive integers, and our investigation to the
set of arithmetic progressions (our structure).  As is common in Ramsey theory, we will use colors to
denote set partition membership.  Formally, for $r \in \mathbb{Z}^+$, an $r$-coloring of the positive
integers is defined by $\chi: \mathbb{Z}^+ \rightarrow \{0,1,\dots,r-1\}$.  We say that
$S \subseteq \mathbb{Z}^+$ is {\it monochromatic} under $\chi$ if $|\chi(S)|=1$. 
In order to
discuss the preservation of structure, and, subsequently, state the 2-Large Conjecture,
we turn to
a fundamental result in Ramsey theory on the integers: van der Waerden's Theorem \cite{vdw}.

\begin{theorem}[van der Waerden's Theorem] For any fixed positive integers $k$ and $r$, every $r$-coloring
of $\mathbb{Z}^+$ admits a monochromatic $k$-term arithmetic progression. \label{vdw}
\end{theorem}

In a certain sense, we cannot break the existence of arithmetic progressions via set partitioning since
van der Waerden's Theorem proves that one of the partition classes {\it must} contain an arithmetic
progression.  If you don't believe me, try $2$-coloring the first nine positive integers without
creating a monochromatic $3$-term arithmetic progression (I'll wait).

So now that we're all on board, the next attribute of arithmetic progressions to take note of is
that they are closed under translation and dilation: if $S=\{a,a+d,a+2d,\dots,a+(k-1)d\}$ is a $k$-term arithmetic
progression, and $b$ and $c$ are positive integers, then $c+bS =\{(ab+c), (ab+c)+bd,(ab+c)+2bd,\dots,(ab+c)+(k-1)bd\}$
is also a $k$-term arithmetic progression.  It is this attribute that affords us a simple inductive argument
when proving van der Waerden's Theorem.  Specifically, assuming that the $r=2$ case of Theorem \ref{vdw}
is true (for all $k$), we may prove that it is true for general $r$ rather simply. 

In order to proceed, we need a restatement of Theorem \ref{vdw}, which is often referred to as the finite version.

\begin{theorem}[van der Waerden's Theorem restatement] For any fixed positive integers $k$ and $r$, 
there exists a minimum integer $w(k;r)$ such that every $r$-coloring
of $\{1,2,\dots,w(k;r)\}$ admits a monochromatic $k$-term arithmetic progression. \label{vdw2}
\end{theorem}
  
The proof of equivalence of Theorem \ref{vdw} and Theorem \ref{vdw2} (at least the nontrivial
direction) is given by {\it The Compactness Principle}, which, in this setting, could also
be called {\it Cantor's Diagonal Principle}, as the proof is an application and slight modification of the
diagonal argument Cantor used to prove that the set of real numbers is uncountable.

Now back to the induction argument.  We may assume that $w(k;s)$ exist for $s=2,3,\dots,r-1$ for any $k\in \mathbb{Z}^+$.
Let $m=w(k;r-1)$ so that $n=w(m;2)$ exists.  Consider $\chi$, an arbitrary $r$-coloring of $\{1,2,\dots,n\}$.  For ease of exposition,
let the colors be red and $r-1$ different shades of blue.  Consider someone who cannot distinguish between shades
of blue so that the $r$-coloring looks like a $2$-coloring to this person.  By the definition of $n$, such a person
would conclude that a monochromatic $m$-term arithmetic progression exists under $\chi$.  If this monochromatic
progression is red, we are done, so we assume that it is ``blue."  Let it be $a+d, a+2d, a+3d,\dots,a+md$ and
note that, since we can distinguish between shades of blue, we have an $(r-1)$-colored $m$-term arithmetic
progression.  We have a one-to-one
correspondence between $(r-1)$-colorings of $T=\{1,2,\dots,m\}$ and $a+dT = \{a+d,a+2d,a+3d,\dots,a+md\}$.
By the definition of $m$ and because arithmetic progressions are closed under translation and dilation, we see that $T$, and hence $a+dT$, admits a monochromatic
$k$-term arithmetic progression, thereby completing the inductive step.

Of course, the previous paragraph is only a partial proof since I made the significant assumption that 
Theorem \ref{vdw} holds for two colors; however, we can state the following:

\begin{quote} \normalsize
($\star$) If every $2$-coloring of $\mathbb{Z}^+$ admits arbitrarily long monochromatic arithmetic progressions, then,
for any $r \in \mathbb{Z}^+$,
every $r$-coloring of $\mathbb{Z}^+$ admits arbitrarily long monochromatic arithmetic progressions. 
\end{quote}
Having obtained this conditional result ($\star$), the rabbit hole is starting to come into view.

Brown, Graham, and Landman \cite{BGL} investigated a strengthening of
Theorem \ref{vdw} by restricting the set of allowable common differences.

\begin{definition}[$r$-large, large, $D$-ap] Let $D\subseteq \mathbb{Z}^+$ and let $r \in \mathbb{Z}^+$.
We refer to an arithmetic progression $a,a+d, a+2d,\dots,a+(k-1)d$ with $d \in D$
as a $k$-term {\it $D$-ap}.
If for any $k \in \mathbb{Z}^+$, every $r$-coloring of $\mathbb{Z}^+$ admits a monochromatic $k$-term 
$D$-ap, then we say
that $D$ is {\it $r$-large} .  If $D$ is $r$-large for all $r \in \mathbb{Z}^+$, then we
say that $D$ is {\it large}.
\end{definition}

Using this definition, we would restate ($\star$) as:  

\begin{quote}\normalsize ($\star$) If $\mathbb{Z}^+$ is $2$-large, then $\mathbb{Z}^+$ is large.
\end{quote}

We now can read the sign above that rabbit hole.  It has the following conjecture, due to
Brown, Graham, and Landman \cite{BGL}, scrawled on it:

\begin{conj}[$2$-Large Conjecture] Let $D \subseteq \mathbb{Z}^+$.  If $D$ is $2$-large, then $D$ is large.
\end{conj}

All known $2$-large sets are also large.  Some  $2$-large sets
are: $m\mathbb{Z}^+$ for any positive integer $m$ (in particular, the
set of even positive integers); the range of any integer-valued
polynomial $p(x)$ with $p(0)=0$; any set $\{\left\lfloor \alpha n \right\rfloor: n \in \mathbb{Z}^+\}$ with $\alpha$ irrational.  We will be
visiting all of these sets on our journey.

As we move forward, you may think you have spotted the rabbit, but that rabbit is
cunning.  Beware of false promise, which comes to you in
hare clothing.

\section{The Carrot}

So, what makes this conjecture so appealing?  Firstly, the $2$-Large Conjecture is
so very natural given the proof of conditional statement ($\star$).  Secondly, there are
several {\it a priori} disparate tools  in Ramsey theory at our disposal.  Thirdly, who doesn't like a challenge;
the lure of the carrot is strong (but don't disregard the stick).

We can approach this problem:
\begin{itemize}
\item[] (1) purely measure-theoretically,
\item[](2) using measure-theoretic ergodic systems, 
\item[](3) using discrete topological dynamical systems, 
\item[](4) algebraically through the Stone-\v{C}ech compactification of $\mathbb{Z}^+$,
and  \item[](5) combinatorically/using other ad-hoc methods.
\end{itemize}
Even though I have described these approaches as disparate, there are  connections between
them that will become clear as we carry on the investigation.

\subsection{Measure-theoretic Approach} On the measure-theoretic front, we must start with Szemer\'edi's \cite{Sze} celebrated result.
For $A \subseteq \mathbb{Z}^+$, let $\bar{d}(A)$ denote the upper density of $A$: $\bar{d}(A) = \limsup_{n \rightarrow \infty} 
\frac{|A \cap \{1,2,\dots,n\}|}{n}$.

\begin{theorem}[Szemer\'edi's theorem]
Any subset $S \subseteq \mathbb{Z}^+$ with $\bar{d}(S)>0$  contains arbitrarily long arithmetic progressions.
\end{theorem}
Szemer\'edi's proof has been called elementary, but it is anything but easy, straightforward, or simple.  In fact, contained with his proof is
a logical flow chart on 24 vertices with 36 directed edges that furnishes the reader with an overview of the intricate web of logic
used to prove the seminal result.

So, how do we mesh this result with $2$-large sets?  Since every $2k$-term arithmetic progression with common difference $d$
contains a $k$-term arithmetic progression with common difference $2d$, we have large sets with positive density (the set
of even positive integers).  A result in \cite{BGL} shows that $\{10^n: n \in \mathbb{Z}^+\}$ is not $2$-large, so we have
sets with 0 density that are not $2$-large.  Perhaps there is a density condition that distinguishes
large and non-large sets.
Unfortunately, further exploration shows this is not true.

We can have sets with positive
upper density that are not $2$-large and we can have sets with zero upper density
that are $2$-large.  To this end, first consider the set of odd integers $D_1$.  Coloring
$\mathbb{Z}^+$ by alternating red and blue, we do not even have a monochromatic $2$-term
$D_1$-ap. Hence, $D_1$ has positive
density but is not $2$-large.  Now consider the set
of squares $D_2$. As a very specific case of a far reaching extension of Szemer\'edi's result,
Bergelson and Liebman \cite{BL} have shown that $D_2$ is large.   More generally
(but still not as general as the full theorem),
Bergelson and Liebman proved the following result.

\begin{theorem}[Bergelson and Liebman] Let
$p(x): \mathbb{Z}^+ \rightarrow \mathbb{Z}^+$ be a polynomial with $p(0)=0$.
Then the set $D=\{p(i): i \in \mathbb{Z}^+\}$
is large.  More precisely, any subset of $\mathbb{Z}^+$ of
positive upper density contains arbitrarily long
$D$-aps.\label{BandL}
\end{theorem}

In quick order we have seen that distinguishing large and non-large sets solely
by their densities is not the correct approach.  
However, the proof of
Theorem \ref{BandL} leads us to our next approach.

\subsection{Measure-theoretic Ergodic Approach} Closely related to the above approach is the use of ergodic systems.  The connection between Szemer\'edi's Theorem
and ergodic dynamical systems is provided by Furstenberg's correspondence principle \cite{Fur}, which uses the following notations.
We remark here that we are specializing all  results to the integers and that the
stated results do not necessarily hold in different ambient spaces; see, e.g., \cite{BM}.

\begin{notation} For $S \subseteq \mathbb{Z}^+$ and $n \in \mathbb{Z}$, we let $S-n = \{s-n: s \in S\}$.
For the remainder of the article, we reserve the symbol $T$ for the shift operator that acts on $\mathcal{X}$, the
family of infinite sequences $x=(x_i)_{i \in \mathbb{Z}}$, by $Tx_n = x_{n+1}$.

\end{notation}

\begin{theorem}[Furstenberg's Correspondence Principle]
Let $E \subseteq \mathbb{Z}^+$ with $\bar{d}(E)>0$.  Then, for any $k \in \mathbb{Z}^+$, there exists
a probability measure-preserving dynamical system $(\mathcal{X},\mathcal{B},\mu,T)$
with a set $A \in \mathcal{B}$ such that $\mu(A) = \bar{d}(E)$ and
$$
\bar{d}\!\left(\bigcap_{i=0}^k (E-in)\right)  \!\geq \!\mu\left(
\bigcap_{i=0}^k T^{-in}A\right)
$$
for any $n \in \mathbb{Z}^+$.
\label{Furst}
\end{theorem}

The above result can be viewed as the impetus for ergodic Ramsey theory as a field of research.
Furstenberg proved that there exists $d \in \mathbb{Z}^+$ such that,
for any $A$ with $\mu(A)>0$ we have
$$\mu\left(A \cap T^{-d}\!A \cap T^{-2d}\!A \cap \!\cdots \! \cap T^{-kd}\!A\right) > 0.$$
By Theorem \ref{Furst}, we have $E \cap (E-d) \cap (E-2d) \cap \!\cdots \!\cap (E-kd) \neq \emptyset$.
Hence, by taking $a$ in this intersection,
we have  $\{a,a+d,a+2d,\dots,a+kd\} \subseteq E$. 
Consequently, Furstenberg provided an ergodic proof of Szemer\'edi's theorem\footnote{Furstenberg used Banach upper density
and not upper density}.

Having followed this path it seems we have hit another dead end in our journey;
there appears to be no mechanism for controlling the number of colors in these arguments.
Perhaps a non-measure-theoretic dynamical system approach can help.

\subsection{Topological Dynamical Systems Approach} As ergodic systems are specific types of dynamical systems, the $2$-Large Conjecture may be
susceptible to the use of a different breed of dynamical
system, namely a topological one.

We will denote the space of infinite
sequences $(x_n)_{n \in \mathbb{Z}}$ with $x_i \in \{0,1,\dots,r-1\}$ by $\mathcal{X}_r$ and let $T$
remain the shift operator acting on $\mathcal{X}_r$.  Specializing to our situation, we
state Birkhoff's Multiple Recurrence Theorem 
due to Furstenberg and Weiss \cite{FW} (see also \cite{Bir}).

\begin{theorem}[Birkhoff's Multiple Recurrence Theorem] \label{th6} Let $k,r \in \mathbb{Z}^+$.
Under the product topology, for any open set $U \subseteq \mathcal{X}_r$ there exists  $d \in \mathbb{Z}^+$ so that
$
U \cap T^{-d}U \cap T^{-2d} \cap \cdots \cap T^{-kd}U \neq \emptyset.
$
\end{theorem}

Furstenberg's and Weiss' result allowed them to give a new proof
of van der Waerden's Theorem: Define a metric for $x,y \in \mathcal{X}_r$ by
$$
d(x,y) = \left(\min_{i \in \mathbb{Z}^+} x(i) \neq y(i)\right)^{-1},
$$
where $x(i)$ denotes the value/color of the $i^{\mathrm{th}}$ positive term in $x$.
A small $d(x,y)$ means we have value/color agreement in the initial terms of $x$ and $y$.
Let $x\in \mathcal{X}_r$ be the sequence corresponding to any given arbitrary $r$-coloring of $\mathbb{Z}^+$.
Theorem \ref{th6} helps prove that there exists  $y \in \{T^mx\}_{m \in \mathbb{Z}^+}$ such that
all of $d(y,T^dy), d(y,T^{2d}y),\dots,  d(y,T^{kd}y)$
are  less than $1$ for some $d$.  
Hence,
$y, T^dy, T^{2d}y, \dots, T^{kd}y$ all have the same first value/color. Since
$y=T^{a}x$ for some $a$ we have $x_a, x_{a+d}, x_{a+2d},\dots,x_{a+kd}$ all
of the same value/color, meaning that $a, a+d,\dots,a+kd$ is a monochromatic
arithmetic progression.

\begin{remark} We can actually have a guarantee that all of $d(y,T^dy),\break d(y,T^{2d}y),\dots, d(y,T^{kd}y)$
are less than any $\epsilon > 0$; however, this is not needed to prove van der Waerden's Theorem.
It does provide for some very interesting results like arbitrary long progressions all
with the same common difference 
each starting in a set of arbitrarily long consecutive intervals.  It should also be remarked
that this latter result can be shown combinatorially, too.
\end{remark}

So how can we use this to attack the $2$-Large Conjecture?  Given a $2$-large set $D$,
we have a guarantee that over the space $\mathcal{X}_2$ there exists $y \in \{T^mx\}_{m \in \mathbb{Z}^+}$ such that
all of $d(y,T^dy), d(y,T^{2d}y),\dots, d(y,T^{kd}y)$ are  less than $1$ for
some $d \in D$.  Our goal is to prove that this criterion implies the same
over the space $\mathcal{X}_r$.

Although Remark 7 states that all of $d(y,T^dy), d(y,T^{2d}y), \dots,\break d(y,T^{kd}y)$
can be arbitrarily small,
we can only guarantee they are less than $1$ (with our given metric) if we require $d$ from a $2$-large
set.  Hence, we could convert an $r$-coloring to a binary equivalent $2$-coloring
if we discovered a result that a long enough $(r-1)$-colored
$D$-ap admits a monochromatic $k$-term $D$-ap (we have no such result, but this idea will prove fruitful in Section 6).

\section{Back Where We Started}
Presently, it seems we have ended up back where we started.  Fittingly, this recurrence phenomenon is a key notion in dynamical
systems.
Momentarily, before getting to the Stone-\v{C}ech compactification, we'll have a diagram to aid in visualizing how  the different types
of results based on the above approaches relate to each other. 

In Figure 1, we give implications between
the types of recurrence we have considered thus far, followed by their definitions.  
\begin{figure}[h]
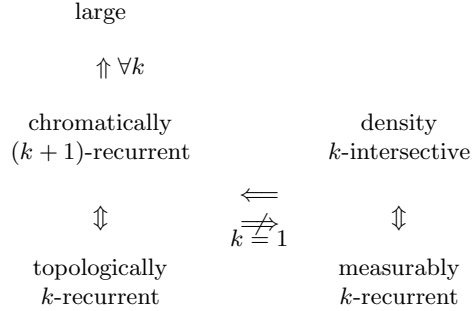
 \large
$$
\begin{array}{ccccccccc}
&\mbox{large}\\[5pt]
&\hskip 15pt\Big\Uparrow{\mbox{\small$\forall k$}}\\[5pt]
&\mbox{chromatically}&\Longleftarrow&\mbox{density}\\
&\mbox{$(k+1)$-recurrent}&{\hskip 5pt{\not} \hskip-7pt\Longrightarrow}&\mbox{$k$-intersective}\\[-3pt]
&&\mbox{\small{$k=1$}}&\\[-5pt]
&
\Big\Updownarrow&&\Big\Updownarrow\\[2pt]
&\mbox{topologically}&&\mbox{measurably}\\
&\mbox{$k$-recurrent}&&\mbox{$k$-recurrent}\\
\end{array}
$$
\caption{Relationship between types of recurrence considered thus far}
\end{figure}

All of the following definitions are given with respect to arithmetic progressions over the integers;
as such, some of the definitions are specific cases of more general definitions.
Some of the implications above fail in more general settings.

\begin{definitions} Let $r \in \mathbb{Z}^+$.  Denote the set of infinite
sequences $(x_n)_{n \in \mathbb{Z}}$ with $x_i \in \{0,1,\dots,r-1\}$ by $\mathcal{X}_r$ and let $T$
be the shift operator acting on $\mathcal{X}_r$.
For $D \subseteq \mathbb{Z}^+$, we say that $D$ is
\begin{itemize}
\item[](i) {\it chromatically $k$-recurrent} if, for any $r$, every
$r$-coloring of $\mathbb{Z}^+$ admits a monochromatic $k$-term arithmetic
progression  $a, a+d,\dots,a+(k-1)d$ with $d \in D$.  
\\
\item[](ii)
{\it topologically $k$-recurrent} if, for any $r$, the dynamical system $(\mathcal{X}_r,T)$
has the property that for
every open set $U \subseteq \mathcal{X}_r$ there exists $d \in D$ such that
$$
U \cap T^{-d}U \cap T^{-2d} \cap \cdots \cap T^{-kd}U \neq \emptyset.
$$
\\[-20pt]
\item[](iii)
{\it density $k$-intersective} if for every $A\subseteq \mathbb{Z}^+$ with $\bar{d}(A)>0$,
there exists $d \in D$ such that
$$
A \cap (A-d) \cap (A-2d) \cap \cdots \cap (A-kd) \neq \emptyset.
$$
\\[-20pt]
\item[](iv) {\it measurably $k$-recurrent} if for any probability measure-preserving dynamical system $(\mathcal{X}_r,\mathcal{B},\mu,T)$
and any $A \in \mathcal{B}$ with $\mu(A)>0$ there exists $d \in D$ such that 
$$\mu\left(A \cap T^{-d}\!A \cap T^{-2d}\!A \cap \!\cdots \! \cap T^{-kd}\!A\right) > 0.$$
\end{itemize}
\end{definitions}

The fact that the double implications shown in Figure 1 are true has already been partially discussed;
see \cite{Jun} for details on the left double implication.
The top-most implication is the definition of large. The negated implication was proved by
K\v{r}\'i\v{z} \cite{Kriz} with nice write-ups by Jungi\'c \cite{Jun} and McCutcheon \cite{McC}, while the remaining implication comes from
the fact that any finite coloring of $\mathbb{Z}^+$ contains a color class of positive
upper density.

The negated implication offers a bit of insight -- a tiny flashlight for our travels, if you will.
The set used to prove this negation is
the set $C'$ given in \cite{McC}:
$$
C' = \{c \in \{0,1,\dots,M\}:  c \!\!\!\! \pmod{p_i} = \pm 1 \mbox{ for at least $2r$ indices $i$}\}
$$
where $r$ is sufficiently large, $p_1, p_2,\dots,p_{2r+2}$ are sufficiently
large primes, and $M=\prod_{i=1}^{2r+2} p_i$.
We do not know if   $C'$   is $2$-large/large or not (this seems to be a
difficult problem in-and-of itself), but if it is then there exists a set of
positive upper density that does not contain a $C'$-ap.
We should take this uncertainty as a warning that we have no guarantee a large set $D$ has its
$D$-aps lie inside a color class with positive upper density, even though
Szemer\'edi's Theorem assures us that arbitrarily long arithmetic progressions are there.

\vskip 20pt
\section{Stone-\v{C}ech Compactification Approach} Although the  three approaches in Section 3 all have nice links between them, the
approach championed by Bergelson, Hindman, Strauss, and others is quite disparate from the previous methods we have seen.  The approach is a blend of
set theory, topology, and algebra.  We'll start by describing the points in the Stone-\v{C}ech compactification
on $\mathbb{Z}^+$, which requires the following definition (again, specialized to the positive integers).

\begin{definition}[filter, ultrafilter]  Let $p$ be a family of subsets of $\mathbb{Z}^+$ (this lowercase $p$ is the standard notation
in this field).  If $p$ satisfies all of
\begin{itemize}
\item[](i) $\emptyset \not\in p$;
\item[](ii) $A \in p$ and $A \subseteq B \Rightarrow B \in p$; and
\item[](iii) $A,B \in p \Rightarrow A \cap B \in p$,
\end{itemize}
then we say that $p$ is a {\it filter}.  If, in addition, $p$ satisfies
\begin{itemize}
\item[](iv) for any $C \subseteq \mathbb{Z}^+$ either $C \in p$ or $C^c=\mathbb{Z}^+ \setminus C \in p$
\end{itemize}
then we say that $p$ in an {\it ultrafilter}.  (Item (iv) means that $p$ is not properly contained in 
any other filter.)
\end{definition}

\begin{example} The set of subsets $\mathcal{F}=\{A \subseteq \mathbb{Z}^+: |\mathbb{Z}^+ \setminus A | < \infty\}$ is a filter but not
an ultrafilter (it is known as the Fr\'echet filter).  It is not an ultrafilter since, taking $C$ from (iv) above to be the set of even positive integers we see
that neither $C$ nor its complement is in $\mathcal{F}$.  The set $\mathcal{G} = \{A \subseteq \mathbb{Z}^+: x \in A\}$ for
any fixed $x \in \mathbb{Z}^+$ is an ultrafilter.
\end{example}

\begin{remark*} One hint that the ultrafilter direction may not prove useful is that the
family of large sets is not a(n) (ultra)filter.  Parts (i), (ii), and (iv) of the
ultrafilter definition are satisfied, but part (iii) is not.
To see this, consider $A = \{i^3: i \in \mathbb{Z}^+\}$ and $B = \{i^3+8: i \in \mathbb{Z}^+\}$.  These are both large sets (see \cite{BGL}); however,
$A \cap B = \emptyset$ (Fermat's Last Theorem serves as a very
useful result for counterexamples) and so cannot be large.
\end{remark*}

The  Stone-\v{C}ech compactification of $\mathbb{Z}^+$ is denoted by $\beta\mathbb{Z}^+$, and
the points in $\beta\mathbb{Z}^+$ are the ultrafilters,
i.e., $\beta\mathbb{Z}^+ = \{p: \mbox{$p$ is an ultrafilter}\}$.  Having the space set, we need to define
addition in $\beta\mathbb{Z}^+$.

\begin{definition}[addition in $\beta\mathbb{Z}^+$]
Let $A \subseteq \mathbb{Z}^+$ and let $p,q \in \beta\mathbb{Z}^+$.  As before,
$A-x =\{y \in \mathbb{Z}^+: y+x \in A\}$.  We define the addition of
two ultrafilters by
$
A \in p+q \Longleftrightarrow \{x \in \mathbb{Z}^+: A-x \in p\} \in q.
$
\end{definition}

The link between $r$-colorings and ultrafilters is provided by the following lemma.

\begin{lemma} Let $r \in \mathbb{Z}^+$ and let $p \in \beta\mathbb{Z}^+$.
For any $r$-coloring of $\mathbb{Z}^+$, one of the color classes is in $p$. \label{9}
\end{lemma}

\begin{proof}
Let $\mathbb{Z}^+ = \cup_{i=1}^r C_i$, where $C_i$ is the $i^{\mathrm{th}}$ color class.
By part (iv) of the definition of ultrafilter, if we assume that none of the $C_i$'s are
in $p$, then each of their complements is in $p$.  Applying part (iii) of the definition
of a filter $r-2$ times, we have $\cap_{i=1}^{r-1} C_i^c \in p$.  But
$\cap_{i=1}^{r-1} C_i^c = C_r$ so $C_r \in p$, a contradiction.
\end{proof}

A   number of results that use $\beta\mathbb{Z}^+$ to give Ramsey-type results rely on
the existence of an additive idempotent in $\beta\mathbb{Z}^+$.  To this end,
if $p+p = p$ is such an element, then $A \in p+p = p$ means that $B=\{x \in \mathbb{Z}^+: A-x \in p\}$ is in $p$
as well.  Thus, we have $a \in A$ and $d \in B$ such that $a$ and $a-d$ are both in A.

Hence, since $p$ is a filter, by item (iii) above we have $A \cap B \in p$.  Hence, appealing to
Lemma \ref{9}, we can consider $A$ to be a color class in an $r$-coloring so that we
have a monochromatic solution to $x+y=z$ with $x=d$, $y=a-d$, and $z=a$.  This result
is known as {Schur's Theorem}.  

In order to obtain van der Waerden's Theorem
through the use of ultrafilters, quite a bit of algebra of $\beta\mathbb{Z}^+$ is needed, so
we will not present that here.  The interested reader should consult the sublime book
by Hindman and Strauss \cite{HS}.  Unfortunately for our goal, while reading through this
book it becomes quite clear that the number of colors used to state the Ramsey-type
results is irrelevant and can be kept arbitrary in the arguments.  So there does not seem a natural
way to control for the number of colors.  On the other hand, there are many
types of {\it largeness} that can be proved using ultrafilters.  
A good description of largeness via ultrafilters is given by Bergelson and Downarowicz in \cite{BD}.  An intuitive notion of largeness would be
that a certain property $P$ is large (in the general sense) if $P$
itself is partition regular or 
if every set with property $P$ is partition regular in some (possibly different) sense.

Perhaps one of
these types of largeness will help achieve our goal.  Figure 2, below, gives
a summary of how these different types are related; most 
of the chart is due to Bergelson and Hindman \cite{BH}.

\begin{figure}[h]
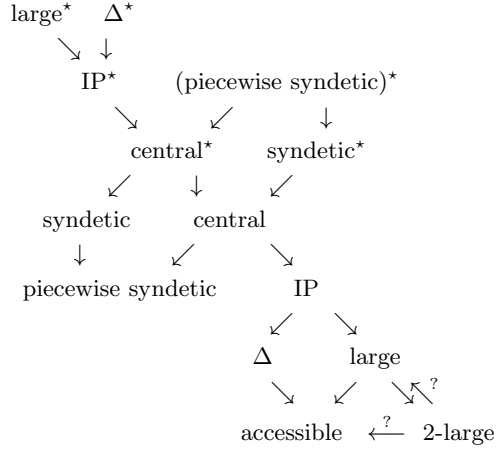
 \large
\begin{center} \hspace*{-55pt}
\begin{tabular}{ccccclcccc}
large$^\star$\hspace*{-55pt}&&$\Delta^\star$\\[2pt]
&$\searrow$\hspace*{-55pt}&\hspace*{-10pt}$\downarrow$\\[2pt]
&&\hspace*{-15pt}IP$^\star$&&&\hspace*{-55pt}(piecewise syndetic)$^\star$\\[2pt]
&&\hspace*{5pt}$\searrow$&&\hspace*{-45pt}$\swarrow$&\hspace*{0pt}$\downarrow$\\[2pt]
&&&\hspace*{-65pt}central$^\star$\hspace*{-12pt}&&\hspace*{-20pt}syndetic$^\star$ \\[2pt]
&&$\swarrow$&\hspace*{-35pt}$\downarrow$&$\swarrow$\\[2pt]
&&\hspace*{-25pt}syndetic&\hspace*{-20pt}central\hspace*{-10pt}\\[2pt]
&&$\downarrow$\hspace*{30pt}&\hspace*{-45pt}$\swarrow$&$\searrow$\\[2pt]
&&piecewise syndetic&&&\hspace*{-10pt}IP\\[2pt]
&&&&$\swarrow$&\hspace*{5pt}$\searrow$\\[2pt]
&&&$\Delta$\hspace*{-15pt}&& &\hspace*{-65pt}large\\[2pt]
&&&\hspace*{10pt}$\searrow$\hspace*{-20pt}&&\hspace*{5pt}$\swarrow$&\hspace*{-55pt}$\searrow$ \hspace*{-5pt}$\nwarrow$ \hspace*{-5pt}{$^?$}\hspace*{-20pt}\\[2pt]
&&&&&\hspace*{-30pt}accessible \hspace*{3pt} $\stackrel{?}\longleftarrow$ \hspace*{3pt}$2$-large\\[10pt]
\end{tabular}
\caption{Type of largeness and implications.  Missing implications are not true.}
\end{center}
\end{figure}

In defining the terms in Figure 2, we will start at the bottom and work our way up.
As we move up the chart, the concepts are all types of largeness that increase
in robustness.

\begin{definitions} Let $S \subseteq \mathbb{Z}^+$.  We say that $S$ is
\begin{itemize}
\item[](i) {\it accessible} if every $r$-coloring, for every $r$, admits arbitrarily long
progressions $x_1,x_2,\dots,x_n$ such that $x_{i+1}-x_i \in S$ for $i=1,2,\dots,n-1$.

\item[](ii) a {\it $\Delta$-set} if there exists $T \subseteq \mathbb{Z}^+$ such that $T-T \subseteq S$.

\item[](iii) an {\it IP-set} if there exists $T = \{t_i\} \subseteq \mathbb{Z}^+$ such that
$FS(T)=\{\sum_{f \in F} t_f: F \subseteq \mathbb{Z}^+ \mbox{ with } |F|<\infty\} \subseteq S$ (the
notation $FS(T)$  stands
for the {\it finite sums} of $T$).

\item[](iv) {\it piecewise syndetic} if there exists $r \in \mathbb{Z}^+$ such
that for any $n \in \mathbb{Z}^+$ there exists
$\{t_1<t_2<\cdots<t_n\} \subseteq S$ with $t_{i+1}-t_i \leq r$ for $1 \leq i \leq n-1$.

\item[](v) {\it syndetic} if there exists $r \in \mathbb{Z}^+$ such that
$S=\{s_1<s_2<\cdots\}$ satisfies $s_{i+1}-s_i \leq r$ for all $i \in \mathbb{Z}^+$.

\item[](vi) a {\it central set} if $S \in p$ where $p$ is an idempotent ultrafilter with
the property that for all $A \in p$, the set $\{n \in \mathbb{Z}^+: A+n \in p\}$ is syndetic.
(The original, equivalent, definition comes from Furstenberg (see \cite{Fur2}) in
the area of dynamical systems.)
\end{itemize}
\end{definitions}

The remaining categories in Figure 2 all have a $^{\star}$ on them.  This is the
designation for the {\it dual} property.  If $X$ is one of the non-starred properties in Figure 2,
then we say that a set $S$ is in $X^\star$ if $S$ intersects every set that
has property $X$.

All implications and non-implications that do not involve any property in
$\{\mbox{$2$-large, large, accessible, large$^\star$}\}$
are from \cite{BH}.  The fact that accessible does not imply large was
first shown by Jungi\'c \cite{Jun}, who provided a non-explicit accessible
set that is not $7$-large.  Recently, Guerreiro, Ruzsa, and Silva \cite{GRS}
provided an explicit accessible set that is not $3$-large. In the next section,
we give an explicit accessible set that is not $2$-large, explaining why
accessible does not imply $2$-large.  The fact that a $\Delta$-set is also
accessible is from \cite[Th. 10.27]{LR}, while the fact that large implies
accessible is straightforward.  The set of cubes is large (and, hence, accessible) but not
a $\Delta$-set. To see this, assume otherwise and let $\{s_i\}$ be
an increasing sequence of positive integers such that
$s_j - s_i$ is a cube for all $j>i$.  Then there exists
$t<u<v$ such that
$s_v - s_t, s_v - s_u,$ and $s_u-s_t$ are all
cubes.  But then we have $s_v-s_t = (s_v-s_u) + (s_u-s_t)$
as an integer solution to $z^3=x^3+y^3$, a contradiction.  Large$^\star$ implying IP$^\star$ comes from
the fact that all IP sets are large sets.  To see that IP$^\star$ does not
imply large$^\star$, let $D$ be the set of integer cubes. Then
$D \not \in $ IP (in fact, for any $x,y \in D$ we have $x+y \not \in D$)\footnote{I have a cute, short proof of this fact, but, like Fermat, there
is not enough room in the bottom margin here, especially given the length of this
lengthy, and unnecessary, footnote.}.  Hence,
any $A \in $ IP cannot be a subset of $D$, meaning that
$A \cap D^c \neq \emptyset$. Thus, $D^c$ intersects all IP-sets;
in other words, $D^c \in $ IP$^\star$.
By Theorem \ref{BandL}, we know that $D$ is large.
Now, because $D^c$ does not interesect
the large set $D$ we see that $D^c \not \in$ large$^\star$.

Investigating some of the properties in Figure 2, we have some interesting
results that could aid in proving or disproving the $2$-large
conjecture.

\begin{theorem}[Furstenberg and Weiss \cite{FW}] Let $k,r \in \mathbb{Z}^+$.  For any $r$-coloring of $\mathbb{Z}^+$,
for some color $i$, the set of common differences of monochromatic $k$-term
arithmetic progressions of color $i$ is in IP$^\star$. \label{IPstar}
\end{theorem}

Theorem \ref{IPstar} gives a very strong property for the set of common differences
in monochromatic arithmetic progressions.  Along these same lines, we have the following result.

\begin{theorem}[Bergelson and Hindman \cite{BH2}] For any $k \in \mathbb{Z}^+$ and any infinite subset of positive integers $A$, under any finite coloring of $\mathbb{Z}^+$
there exists a monochromatic $k$-term arithmetic progression with common difference
in $FS(A)$.
\end{theorem}

Unfortunately, neither of these last two theorems helps guide us out of
the rabbit hole, as generic $2$-large sets don't necessarily have any obvious structures.
So we now find ourselves moving into the darkest corner
of the hole: the combinatorics encampment.

\section{Some Combinatorial Results}

We start this section by providing an accessible set that is not $2$-large.  Had this result
not been possible, we would have disproved the $2$-Large Conjecture (since there exists an accessible
set that is not large) and been saved from the
depths of the rabbit hole.  But, alas, it was not meant to be.

\begin{theorem} There exists an accessible set that is not $2$-large.
\end{theorem}

\begin{proof} This proof takes as inspiration the proof from \cite{GRS}, which provides an accessible set
that is not $3$-large.  Let $S=\{2^{4i}: i \geq 0\}$.
From \cite[Th. 10.27]{LR}, we know that $S-S = \{2^{4j} - 2^{4i}: 0 \leq i < j\}$
is an accessible set.  We will provide a $2$-coloring of $\mathbb{Z}^+$ that avoids monochromatic $25$-term 
$(S-S)$-aps.

For each $n \in \mathbb{Z}^+$, write $n$ in its binary representation 
so that $\sum_{i \geq 0} b_i(n) 2^i = n$. Next, partition the
$b_i=b_i(n)$ into intervals of length $4$:
$$
[\dots b_i b_{i-1} \dots b_2 b_1 b_0] = \bigcup_{j \geq 0} I_j(b),
$$
where $I_j(b) = [b_{4j+3}, b_{4j+2}, b_{4j+1}, b_{4j}]$.  Each of these $I_j(b)$'s is one of
16 possible binary sequences (if some or all of $b_{4j+3}, b_{4j+2},$ and $b_{4j+1}$ are
missing for the largest $j$, we take the missing terms to be $0$).  Apply the
mapping $m$ to each $I_j(b)$:
$$
m:   \left\{ \hspace*{-5pt}
\begin{array}{lr}
 [0,0,1,1], [0,1,a,b],  [1,0,0,0], [1,0,1,0], [1,0,1,1] \! \longrightarrow \!0\\[5pt]
 [0,0,0,0], [0,0,0,1], [0,0,1,0], [1,0,0,1], [1,1,c,d] \!\longrightarrow \! 1,\\
\end{array}
\right.
$$
\normalsize
where $a,b,c,d$ may be, independently, either $0$ or $1$.

\normalsize
Let $m_j(b) = m(I_j(b))$.  We color the integer $n$ by $\chi(n) = \sum_{j} m_j(b) \,\,(\mbox{mod } 2)$.
We will now show that $\chi$ is a $2$-coloring of $\mathbb{Z}^+$ that does not admit a
monochromatic $25$-term arithmetic progression with common difference from $S-S$.

Let $x_1 < x_2 < \cdots < x_{25}$ be an arithmetic progression with common difference
$d=2^{4s} - 2^{4t}$.  We will use the shorthand $I_j(\ell)$ to represent $I_j(x_{\ell})$
and $m_j(\ell)$ to represent $m(I_j(\ell))$.
First, consider $U=\{I_t(\ell): 8 \leq \ell \leq 23\}$.  By definition of $d$, we see
that $I_t(\ell)$ and $I_t(\ell+1)$ will differ; moreover, by the definition of $d$, the set
$U$ will contain all 16 possible binary strings of length 4.  In particular, there exists
$r \in \{8,9,\dots,23\}$ such that $I_t(r) = [0,0,1,0]$.  Thus, by adding/substracting multiples
of $d$, we can conclude the following:
$$
\begin{array}{c|c|c}
j&I_t(j)&m(I_t(j))\\[3pt]\hline
&\\[-7pt]
r-7&[1,0,0,1]&1\\[3pt]
r-6&[1,0,0,0]&0\\[3pt]
r-5&[0,1,1,1]&0\\[3pt]
r-4&[0,1,1,0]&0\\[3pt]
r-3&[0,1,0,1]&0\\[3pt]
r-2&[0,1,0,0]&0\\[3pt]
r-1&[0,0,1,1]&0\\[3pt]
r&[0,0,1,0]&1\\[3pt]
r+1&[0,0,0,1]&1\\[3pt]
r+2&[0,0,0,0]&1
\end{array}
$$
\normalsize
(we do not need $j \in \{r+3, r+4,\dots,r+8\}$ and these need not exist).

Next, we consider how $d$ affects $I_s(\ell)$ for $\ell \in \{r-7, r-6, \dots,r-1,r+1, r+2\}$ for all
possible cases of $I_s(r)$.  If $I_s(r) = [1,0,0,0]$ we have $m_s(r)=0$ so that
$m_s(r)+m_t(r) = 1$.  This gives us $I_s(r+1) = [1,0,0,1]$ so that
$m_s(r+1)+m_t(r+1) = 2$.  Since all other $I_j(r)$ are unaffected by the addition of $d$
(to get to $I_j(r+1)$; there are no carries with addition by $d$).  Hence, $\chi(x_r) \neq \chi(x_{r+1})$.

This same analysis (perhaps switching the values of the sums $m_s(j)+m_t(j), j=r,r+1$) 
holds when $I_s(r) \in \{[1,0,0,1], [1,0,1,1]\}$.  If $I_s(r)$ is in $\{\,[1,1,0,1],\,\, [1,1,1,0], [1,1,1,1],\,\, 
[0,0,0,1],\,\,  [0,0,1,0],\,\,  [0,1,0,0],\break  [0,1,0,1],$   $[0,1,1,0], [0,1,1,1]\,\}$ then we have
$\chi(x_r) \neq \chi(x_{r-1})$ easily since there are no carry issues.
If $I_s(r) = [1,0,1,0]$ then $\chi(x_r) \neq \chi(x_{r+2})$; if $I_s(r) = [1,1,0,0]$ then
$\chi(r) \neq \chi(x_{r-3})$.  The remaining two cases $I_s(r) \in \{[0,0,0,0], [0,0,1,1]\}$ both
involve carries and need extra attention.  

If $I_s(r)=[0,0,0,0]$ then either
$\chi(r) \neq \chi(r-4)$ or $\chi(r) \neq \chi(r-5)$ depending on whether or not
any carries change the value of $\sum_{j > s} m_j(r-1)$.  Lastly,
if $I_s(r) = [0,0,1,1]$, then either
$\chi(r) \neq \chi(r-6)$ or $\chi(r) \neq \chi(r-7)$ depending on whether or not
any carries change the value of $\sum_{j > s} m_j(r-1)$.

We have shown that for any possible $I_s(r)$, the $25$-term $(S-S)$-ap cannot be monochromatic, thereby
proving the theorem.
\end{proof}

We will now present some positive results.  But first, some definitions.

\begin{definition}[$r$-syndetic] Let $S=\{s_1<s_2<\cdots\}$ be a syndetic set
with $s_{i+1}-s_i \leq r$ for all $i \in \mathbb{Z}^+$.  Then we say that $S$ is {\it $r$-syndetic}
\end{definition}

\begin{definition}[anastomotic] Let $D \subseteq \mathbb{Z}^+$.  If every syndetic set admits a $k$-term $D$-ap
then we say that $D$ is {\it $k$-anastomotic}.  If $D$ is $k$-anastomatic for all $k \in \mathbb{Z}^+$, then we
say that $D$ is {\it anastomotic}.
\end{definition}

\begin{remark*}  The term syndetic has been defined by other authors and is an adjective meaning serving
to connect.  The term anastomotic is new and I chose it since its meaning is: serving to communicate
between parts of a branching system.
\end{remark*}

\begin{theorem} Let $D \subseteq \mathbb{Z}^+$.  If $D$ is $r$-large, then every $r$-syndetic set admits
arbitrarily long $D$-aps.  Conversely, if every $(2r+1)$-syndetic set admits arbitrarily long
$D$-aps, then $D$ is $r$-large. \label{13}
\end{theorem}

\begin{proof} The first statement is straightforward: for any $r$-syndetic set $S$, we define an $r$-coloring
$\chi:\mathbb{Z}^+\rightarrow \{1,\dots,r\}$ by $\chi(i) = \displaystyle \min(\{s-i: s \in S \mbox{ with }i \leq s\}+1$.  Since
$D$ is $r$-large we have arbitrarily long $D$-aps under $\chi$. Since arithmetic progressions are translation invariant, 
by the definition of $\chi$ we see that $S$ admits
arbitrarily long $D$-aps.

To prove the second statement, consider $\sigma$, an arbitrary $r$-coloring of $\mathbb{Z}^+$ using the colors $1,2,\dots,r$.
For every color $i$, replace each occurrence of the color $i$ in $\sigma$ by the string of length $r$ with a $1$ in the $i$th position and
$0$ in all others.  This process gives us a $2$-coloring $\hat{\sigma}$ of $\mathbb{Z}$.
Let $S$ be the set of positions of all $1s$ under $\hat{\sigma}$.  Note that $S$ is a $(2r+1)$-syndetic set.
By assumption, $S$ admits arbitrarily long monochromatic $D$-aps.  In particular, it admits arbitrarily long
$(rD)$-aps (by taking every $r$th term in a sufficieintly long $D$-ap).
Since we now have an arbitrarily long $(rD)$-ap, in the original $r$-coloring $\sigma$ we have a monochromatic $D$-ap.
\end{proof}

\begin{remark*} Recently, Host, Kra, and Maass \cite{HKM} have independently proven a result similar
to Theorem \ref{13}; their result is slightly stronger in that they prove that ``$(2r+1)$-syndetic"
can be replaced by ``$(2r-1)$-syndetic."
\end{remark*}

An immediate consequence of Theorem \ref{13} is the following:

\begin{corollary} Let $D \subseteq \mathbb{Z}^+$.  Then $D$ is large if and only if
$D$ is anastomotic. \label{cor14}
\end{corollary}

It would be nice if every syndetic set contained an infinite
arithmetic progression as we would be done: since every $2$-large
set contains a multiple of every integer, every $2$-large set would
be anastomotic and, by Corollary \ref{cor14} we would be done.  Unfortunately, this is not true.

Let $\alpha$ be irrational and let $\chi: \mathbb{Z}^+ \rightarrow \{0,1\}$
be
$\chi(x) = \lfloor \alpha n \rfloor - \lfloor \alpha (n-1) \rfloor$.
Each color class under $\chi$ corresponds to a syndetic set
(e.g., if $\alpha$ is the golden ratio, then each color
class is $4$-syndetic).  Consider $S = \{i: \chi(i)=0\}$.

Assume, for a contradiction, that there exist integers
$a$ and $d$ such that $S$ contains $a+dn$ for all $n\in\mathbb{Z}^+$.
Then we have $\lfloor \alpha(a+dn)\rfloor = \lfloor \alpha(a+d(n-1))\rfloor$
for all positive integers $n \geq 2$.  Let $\beta = \alpha d$ and note
that $\beta$ is also irrational.  Let $\{ x \}$ be the fractional part
of $x$.  Then $\{ \{\beta n\}: n \in \mathbb{Z}^+\}$ is dense in
$[0,1)$. 

Consider $y \in (\{-\alpha a\}, \{\alpha - \alpha a\})$.  We claim that
we cannot have $\{\beta n\} = y$ for any $n \in \mathbb{Z}^+$.
Assume to the contrary that $\{\beta j\} = y$ so that
$\beta j = \ell + y$ for some integer $\ell$.  Then
$\beta j - y$ is an integer.  By choice of $y$, we have
$\beta j - y$ strictly between $\beta j + \alpha a - \alpha$ and 
$\beta j + \alpha a$.  But this is not possible since
$\lfloor \beta j + \alpha a - \alpha\rfloor = \lfloor \beta j + \alpha a \rfloor$.  Hence, 
$\{ \{\beta n\}: n \in \mathbb{Z}^+\} \cap (\{-\alpha a\}, \{\alpha - \alpha a\}) = \emptyset.$ But 
$\{ \{\beta n\}: n \in \mathbb{Z}^+\}$ is dense in
$[0,1)$, a contradiction.\footnote{The preceeding argument
is based on an answer given by Mario Carneiro on
{\tt math.stackexchange.com} for question 1487778; I was unable
to find a published reference.}

Hence, we have a syndetic set without an infinite arithmetic
progression.  The same analysis shows that $\{i: \chi(i)=1\}$
does not contain one either.  Hence, we can cover
the positive integers with two syndetic sets, neither of which
contain an infinite arithmetic progression.

\subsection{A Brief Detour}

We take a brief side trip back to the dynamical system setting 
in order to expand on Figure 1 
to include the new notions just introduced so that we have an
overview of how the different types of recurrence are related. Furthermore,
we define two other types of recurrence to display the relationships between a fixed number of
colors and an arbitrary number of colors.

\begin{definitions}
Let $r \in \mathbb{Z}^+$ be fixed and consider $D \subseteq \mathbb{Z}^+$.
 We say $D$ is {\it $r$-chromatically $k$-recurrent} if every
finite $r$-coloring of $\mathbb{Z}^+$ admits a monochromatic $k$-term $D$-ap; we say $D$ is 
{\it $r$-syndetically $k$-recurrent} if every $r$-syndetic set contains a $k$-term $D$-ap.
\end{definitions}

In Figure \ref{fig3} below, we assume that the same restrictions as those in Figure 1
are still in place. Missing implications are unknown.

\begin{figure}[h]
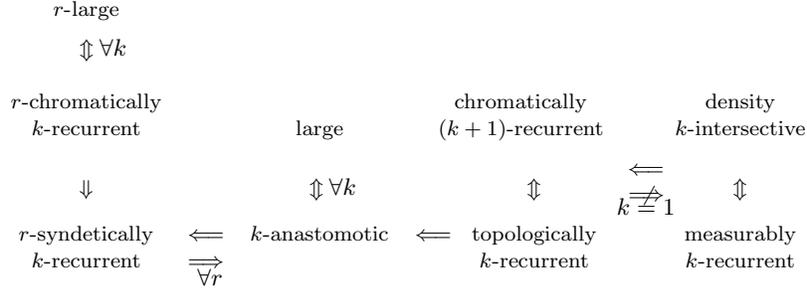
 
$$
\begin{array}{ccccccccc}
&&&&\mbox{measurably}\\
&&&&\mbox{$k$-recurrent}\\
&&&&\Big \Updownarrow\\
&&&&\mbox{density}\\
\mbox{$r$-large}&&&&\mbox{$k$-intersective}\\[5pt]
\hskip 13pt\Big\Updownarrow{\mbox{\small$\forall k$}}&&&&\hskip 19pt\Big \Uparrow \hskip 5pt{\not} \Big\Downarrow \mbox{\small $k\!=\!1$}\\[5pt]
\mbox{$r$-chromatically}&&&& \mbox{chromatically}&&\\
\mbox{$k$-recurrent}&&\mbox{{large}}&& \mbox{$(k+1)$-recurrent}&&\\
\Big\Downarrow&&\hskip 10pt\Big\Updownarrow {\mbox{\small$\forall k$}}&&
 \Big\Updownarrow&&\\
\mbox{$r$-syndetically}&\hspace*{-5pt}\Longleftarrow&\mbox{$k$-anastomotic}&\Longleftarrow \hspace*{-10pt}& \mbox{topologically}&&\\
\mbox{$k$-recurrent}&\hspace*{-5pt}\Longrightarrow&&& \mbox{$k$-recurrent}&&\\
[-3pt]
&\hspace*{-10pt}{\small \mbox{ $\forall r$}}\\[-10pt]
\end{array}
$$ \normalsize
\caption{Relationship between types of recurrence considered in this article}\label{fig3}
\end{figure}

\subsection{Back to the Combinatorics Encampment}

You are surely asking yourself, if you've traveled with me this far: do you
have any positive results? Well, to help aid you in keeping a sanguine
outlook, I'll now offer a few items that offer a
glimmer of hope for escape from our current residence in the dark depths 
of the large rabbit hole (pun intended).

We start with a  strong condition for a set to be large.

\begin{definition}[bounded multiples condition]
We say that $D\subseteq \mathbb{Z}^+$ satisfies the {\it
bounded multiples condition} if there exists $M \in \mathbb{Z}^+$
such that for every $i \in \mathbb{Z}^+$, there exists $m \leq M$ such
that $im \in D$.  In other words, for every positive integer $i$,
at least one element of $\{i, 2i, 3i,\dots,Mi\}$ is in $D$. \label{bmc}
\end{definition}

Using this definition, we have the following characterization.

\begin{theorem} If $D \subseteq \mathbb{Z}^+$ satisfies the bounded multiples
condition, then $D$ is large. \label{prop1}
\end{theorem}

\begin{proof} Let $M$ be the constant that exists by the bounded multiples condition.
We proceed by showing that $D$ is $k$-anastomotic
for all $k$.  Let $A \subseteq \mathbb{Z}^+$ by syndetic.  If $A^c$
is not syndetic, then it has arbitrarily long gaps.  This means
that $A$ contains arbitrarily long intervals.  In this situation,
$A$ contains arbitrarily long $D$-aps. 

Now let both of the sets $A=\{a_i\}_{i \in \mathbb{Z}^+}$ and $A^c=B=\{b_i\}_{i \in \mathbb{Z}^+}$  be $g$-syndetic with
$g = \max_{i \in \mathbb{Z}^+}\{a_{i+1}-a_i, b_{i+1}-b_i\}$.  Let
$\chi(n)$ equal 1 if $n \in A$ and $0$ if $n \in B$.  Define
the $2^{g+1}$-coloring  $\gamma: \mathbb{Z}^+ \rightarrow \{0,1\}^{g+1}$
by $\gamma(n) = (\chi(n), \chi(n+1),\dots,\chi(n+g))$.
Note that $\gamma(n)$ cannot consist of all 0s or all 1s by
the definition of $g$.

Since $\gamma$ is a finite coloring of $\mathbb{Z}^+$, by van der Waerden's
Theorem there exists a monochromatic $Mk$-term arithmetic progression
under $\gamma.$  By the definition of $\gamma$ both $A$ and $B$
contain $Mk$-term arithmetic progressions, each with the same
common difference.  Let $d$ be this common difference.

Since $D$ satisfies the bounded multiples condition, there exists $m \leq M$ such that
$md \in D$.  By taking every $m^{\mathrm{th}}$ term of the $Mk$-term arithmetic progressions,
we see that both $A$ and $B$ have $k$-term $D$-aps.
\end{proof}

\begin{remark*} The converse of Theorem \ref{prop1} is not true.
We know that the set of perfect squares is large; however, it  
 does not satisfy the bounded multiples
condition.  To see this, consider a prime $p$.  Then the smallest multiple
of $p$ in the set of perfect squares is $p^2$.  Since $p$ may be arbitrarily
large, we do not have the existence of $M$ needed in the definition above.
\end{remark*}

As was done at the beginning of this article, we will now present a
``finite version" of the 2-large definition.  Instead of appealing
to the Compactness Principle, we will offer a terse proof of
equivalence.

\begin{lemma} Let $D$ be $2$-large.  For each $k \in \mathbb{Z}^+$, there exists 
an integer $N=N(k,D)$ such that every $2$-coloring of $\{1,2,\dots,N\}$ admits
a monochromatic $k$-term $D$-ap. \label{lem16}
\end{lemma}

\begin{proof} Assume not and, for each $i \in \mathbb{Z}^+$, let $\chi_i$ be a $2$-coloring of $\{1,2,\dots,i\}$ with
no monochromatic $k$-term $D$-ap.  Define $\gamma$, inductively, by
$\gamma(j)=c_j$ where $\chi_i(j)=c_j$ occurs infinitely often among those $\chi_i$
where
$\chi_i(\ell) = \gamma(\ell)$ for $\ell<j$.  Now note that $\gamma$ is a
$2$-coloring of $\mathbb{Z}^+$ with no monochromatic $k$-term
$D$-ap, a contradiction.
\end{proof}

We will have use for the following notation in the remainder of this section.

\begin{notation}  Let $m \in \mathbb{Z}^+$ and $D \subseteq \mathbb{Z}^+$.
Then $mD = \{md: d \in D\}$ and $D^m = \displaystyle\left\{ \prod_{f \in F}  d_{f}: d_{i}  \in D, F \subseteq \mathbb{Z}^+ \mbox{ with } |F|=m
\right\}$. 
\end{notation}

We can now present two easy lemmas.  We will provide a proof for the
first and leave the very similar proof of the second to the reader.

\begin{lemma} Let $D$ be $2$-large and let $m \in \mathbb{Z}^+$.  Then $N(k,mD) \leq mN(k,D)$. \label{lem17}
\end{lemma}

\begin{proof} Consider any $2$-coloring of the first $mN(k,D)$ positive
integers.  We will show that there exists a monochromatic $k$-term
$mD$-ap.  Given our coloring, consider only those integers divisible
by $m$.  Via the obvious one-to-one correspondence between
$\{m,2m,\dots,mN(k,D)\}$ and $\{1,2,\dots,N(k,D)\}$, we have a monochromatic
$k$-term $D$-ap in the latter interval, meaning that we have a monochromatic
$k$-term $mD$-ap in the former interval.
\end{proof}

\begin{lemma}
Let $m\in \mathbb{Z}^+$.  Then $D$ is $r$-large if and only if $mD$ is $r$-large. \label{lem18}
\end{lemma}

We also will use the following definition.

\begin{definition} Let $D$ be $2$-large. Define $M(k,D;2) = N(k,D)$ and, for $r \geq 3$,
$$
M(k,D;r)=N(M(k,D;r-1),D),
$$
where $N(k,D)$ is the integer from Lemma \ref{lem16}.
\end{definition}

As $N(k,D)$ exists for all $k$ (by Lemma \ref{lem16}), we see that $M(k,D;r)$ is well-defined
for all $k$ and $r$.

We are now ready to finally see a ray of hope in the next theorem. 
due to
Brown, Graham, and Landman \cite{BGL}.  The proof presented here is
the details of the sketch given in \cite{BGL} using the functions defined in
the present article. 

\begin{theorem}[\cite{BGL}] Let $r \in \mathbb{Z}^+$. Let $D \subseteq \mathbb{Z}^+$ be $r$-large.
Then $D^m$ is $r^m$-large. \label{main}
\end{theorem}

We will prove this for $r=2$.  The generalization to
arbitrary $r$ is clear but requires new notations and
definitions.

\begin{proof}  The proof is by induction on $m$. We will show that 
any $2^m$-coloring of
$[1,M(k,D^m;2^m)]$ admits a monochromatic $k$-term $D^{m}$-ap.  
For $m=1$, we have $M(k,D^m;2^m) = N(k,D)$ so that the base case
holds by definition of $N(k,D)$.
We now assume the statement for $m$ and will show it for $m+1$.
Note that we will be using the definition
$M(k,D^m;r) = N(M(k,D^m;r-1),D)$ as the function $N$ is only
valid for the given $2$-large set $D$.
Consider any $2^{m+1}$-coloring of $[1,M(k,D^{m};2^{m+1})]$.
Partition the colors into two equal sets, each of size $2^m$.
Calling these sets ``red" and ``blue," we can view the  coloring
as a 2-coloring of the interval $[1,N(M(k,D^m; 2^{m+1}-1),D)] \supseteq [1,N(M(k,D^m; 2^{m}),D)]$.  By definition of $N(M(k,D^m; 2^{m}),D)$ we  have
either a ``red" or ``blue" $D$-ap of length $M(k,D^m; 2^{m})$.
In either case, we have a $2^m$-coloring of a $D$-ap:
$a+d[1,M(k,D^m; 2^{m})]$.  By the induction hypothesis, 
$[1,M(k,D^m; 2^{m})]$ admits a monochromatic $D^m$-ap of length $k$, giving
us a monochromatic $k$-term $D^{m+1}$-ap. 
\end{proof}

\section{Epilogue}
Having attained Theorem \ref{main}, a hidden door in the rabbit hole has opened, leading
us back to the Stone-\v{C}ech locale; by Theorem \ref{main}, if $D$ is
a $2$-large subsemigroup of $(\mathbb{Z}^+,\cdot)$,
then $D$ is large.  
For example, this result gives us, for any monomial $x^n$:   if
$S=\{x^n: x \in \mathbb{Z}^+\}$ is $2$-large, then it is large.
This holds 
since $i^n \cdot j^n = (ij)^n$ gives us that $S$ is a
semigroup (of course, we already know $S$ is large via other means).  However, the set of odd positive integers is also
a semigroup (under multiplication) but is not $2$-large since, from \cite{BGL}, a $2$-large
set must have a multiple of every positive integer.  
Combining the range of polynomials and multiples of every positive integer,
 Frantzikinakis \cite{Fran} has shown that if $p(n): \mathbb{Z}^+ \rightarrow
\mathbb{Z}^+$ is an integer-valued polynomial then $D=\{p(n): n \in \mathbb{Z}^+\}$
is measurably $k$-recurrent for all $k$ if and only if it contains
multiples of every positive integer.  

And now we once again find ourselves wading in the dynamical
systems pool after traveling through the Stone-\v{C}ech locale.

Note that  Frantzikinakis' result
is stronger than  $D$ being large (see Figure \ref{fig3}) and relies heavily on polynomials but
also suggests that, perhaps, if $D$ contains
multiples of every positive integer, then $D$ is large. 
If this were true, then  the $2$-Large Conjecture is true.
But yet again, we are foiled: the set $\{n! : n \in \mathbb{Z}^+\}$
clearly contains a multiple of every positive integer, but
is not  2-large \cite{BGL}.

And now we are back in the combinatorics encampment.
  
\vskip 5pt
Okay silly rabbit, enough tricks; I surrender.  
\vskip 5ptFor now.

\vskip 20pt
\noindent
{\bf Acknowledgments.} I would like to thank
Dan Saracino for a very careful reading of an earlier version of
this paper.
I would also like to thank the anonymous referee for an incredibly helpful and detailed report; this paper has been made better because of it.
Finally, thank you to Sohail Farhangi for finding a gap in a proof of
a theorem that appeared in a previous version of this paper.
%I would also like to publicly express my sincere gratitude to my PhD advisor, Doron %Zeilberger, for
%giving me the inspiration to stay in graduate school, as well
%as for having the insight to give me a dissertation topic in Ramsey theory.

\end{document}